\documentclass[arxiv,reqno,twoside,a4paper,11pt]{amsart}

\usepackage{upref,amsmath,amssymb,amsthm,mathrsfs,amsfonts}
\usepackage[colorlinks=true,linkcolor=blue,citecolor=blue]{hyperref}

\usepackage{tools}
\usepackage{local} 
\usepackage{math62,thm10-REAR}

\usepackage[scaled]{helvet} 
\usepackage{courier} 
\usepackage[mathbf]{euler}

\numberwithin{equation}{section}

\begin{document}

\date{\today}

\title[Anomaly terms in Cheeger-M\"uller theorem]{On the anomaly terms in 
Cheeger-M\"uller theorem on Spaces with conical singularities}

\author{Luiz Hartmann}
\address{Departamento de Matem\'atica, 
	Universidade Federal de S\~ao Carlos (UFSCar),
	Brazil}
\email{hartmann@dm.ufscar.br, luizhartmann@ufscar.br}
\urladdr{http://www.dm.ufscar.br/profs/hartmann}
\thanks{Partially support by FAPESP (2021/09534-4).}

\author{Mauro Spreafico}
\address{Dipartimento di Matematica e Applicazioni, Universit\`a Degli Studi 
di Milano}
\email{mauro.spreafico@unimib.it}


\maketitle

\begin{abstract} 	We prove that the anomaly term appearing in the ratio between Reidemeister torsion and analytic torsion for a space with a conical singularity vanishes in the smooth case. Moreover, we show that such anomaly term is non trivial for the cone over a torus. 
\end{abstract}

\tableofcontents

\section{Introduction}
\label{s.Intro}

The Cheeger-M\"uller theorem (\cite{Che0,Mul1})  affirms the equivalence between the Reidemeister torsion and the analytic torsion for closed  Riemannian manifolds. Among several extension of this result, recently that on spaces with conical singularities has been investigate in \cite{Lud2,HS6}).  Differently from the smooth case, in the singular case
an anomaly term appears, i.e. the ratio between Reidemeister torsion and analytic torsion is not trivial. A first natural question is then  if this anomaly disappears in the smooth case. A second question is to provide an explicit example where such anomaly is non trivial.  This work is devoted to answer these questions. In Section \ref{s.smothcae}, we show that the anomaly vanishes in the smooth case, while in Section \ref{s.Nontriviality}, we prove that the anomaly is non trivial for the cone over a torus.

\section{Cheeger-M\"uller theorem in spaces with conical singularities}
\label{ss.CMtheorem}

Let $(W,\tilde{g})$ be a closed Riemannian manifold with metric $\tilde{g}$. 
The {\it finite metric cone} over $W$ is defined by $C_{0,l}(W):= [0,l] 
\times W/ \{0\}\times W$, where $l>0$, with metric $dr + r^2\tilde{g}$ for 
$r>0$. The tip of the cone is denoted by $\{v\}$. Consider a 
closed pseudomanifold $X$ with a conical singularity, \ie there exists a 
point $x_0\in X$ and a neighbourhood $U$ of $x_0$ such that $X-\{x_0\}$ is a 
Riemannian manifold (we will denote the metric by $g$) and 
$\overline{U}-\{x_0\}$ with the induced metric
is isometric to the open metric cone $C_{0,l}(W)-\{v\}$. We 
can decompose $X = M\cup C_{0,l}(W)$, where $M$ is a 
smooth Riemannian manifold with boundary $\b M = W$ and the union is along  
the boundary. Let $\rho:\pi_1(X)\to O(\R^n)$ be an orthogonal representation, and $E_\rho$ the associated flat vector bundle over $X$. In this context, an extension of the Cheeger-M\"uller theorem reads:
\begin{theorem}[Theorem 8.6.1 \cite{HS6}]\label{theo.CMSm}
	If $\|\cdot\|^{\rm RS}_{\det 
		I^{\mf}H_\bu (X;E_\rho)}$ and $\|\cdot\|^{\rm IR}_{\det I^{\mf}H_\bu 
		(X;E_\rho)}$ are the Ray-Singer metric and Intersection 
	Reidemeister metric defined on the intersection homology line bundles 
	with middle perversity, respectively, then
	\begin{equation}
		\log \frac{\|\cdot\|^{\rm RS}_{\det 
				I^{\mf}H_\bu (X;E_\rho)}}{\|\cdot\|^{\rm IR}_{\det I^{\mf}H_\bu 
				(X;E_\rho)}} = A_{\rm comb,\mf}(W)+A_{\rm analy}(W).
	\end{equation}  The  combinatorial anomaly 
	term, $ A_{\rm comb,\mf}(W)$, and the  analytic anomaly term, $A_{\rm analy}(W) 
	$, are equal to zero if the dimension 
	of $X$ is even, while if the dimension of $X$ is odd, $\dim X=2p+1$ ($p\geq 1$): 
			\begin{equation}
			A_{\rm comb, \mf}(W):=\sum_{q=0}^{p} (-1)^{q+1} \log \Bl\#TH_q(W;\Z) 
			\cdot \left|\det(\tilde{\mathcal{A}}_q(\tilde{h}_q)/n_q) \right| \Br,
		\end{equation}
		where $\#TH_q(W;\Z) $ is the rank of the torsion subgroup of $H_q(W;\Z)$, $\tilde{h}_q$ is an orthonormal base of the harmonic forms in dimension $q$, $n_q$ the standard base of $H_q(W;E_{\rho}|_W)$,  $\tilde{\mathcal{A}}_q$ is the Ray-Singer isomorphism (see \cite[Eq. (2.2)]{HS6} or \cite[Eq. (3.5)]{RS}) and $(\tilde{\mathcal{A}}_q(\tilde{h}_q)/n_q) $ denotes the matrix of the change of basis in $H_q(W;E_{\rho}|_W)$;

		\begin{align}
			A_{\rm analy}(W)&:= \sum_{q=0}^{p-1} (-1)^{q+1}r_q \log(2p-2q-1)!! 
			+\frac{1}{4} \chi(W;E_\rho |_W)\log 2 \label{eq.anomalyanaly.1}\\
			&+\frac{1}{2}\sum_{q=0}^{p-1} 
			(-1)^q \sum_{n=1}^\infty m_{{\rm cex},q,n}\log 
			\frac{\mu_{q,n}+\alpha_q}{\mu_{q,n}-\al_q}
			,\label{eq.anomalyanaly.2}
		\end{align}
		where  $\alpha_q:= \frac{1}{2}+q-p$, 
		$\mu_{q,n}:=\sqrt{\lambda_{q,n}+\alpha_q^2}$, 
		$\{\lambda_{q,n}\}_{n=1}^\infty$ is the set eigenvalues associated to 
		co-exact $q$-dimensional eigenforms of $W$, $r_q = \dim H_q(W;E_{\rho}|_W)$ and
		$\chi(W;E_\rho |_W)$ is the Euler characteristic of $W$.
	
\end{theorem}

The sum in \Eqref{eq.anomalyanaly.2} can be rewritten in a different way. By 
the Hodge 
duality $\mu_{q,n} = \mu_{2p-q-1,n}$, $m_{{\rm cex},q,n} = m_{{\rm 
		cex},2p-1-q,n}$ and  $\alpha_q=-\al_{2p-1-q}$, whence 
\begin{equation}\label{eq.anomalyanaly.3}
	\frac{1}{2}\sum_{q=0}^{p-1} 
			(-1)^q \sum_{n=1}^\infty m_{{\rm cex},q,n}\log 
			\frac{\mu_{q,n}+\alpha_q}{\mu_{q,n}-\al_q}
=\sum_{q=0}^{2p-1} 
	(-1)^q \sum_{n=1}^\infty m_{{\rm cex},q,n}\log(\mu_{q,n}+\al_q).
\end{equation}

In the following, it will be useful to recall  that the infinite sum in equation \Eqref{eq.anomalyanaly.2} stands for the following analytic extension
\begin{equation*}
	\begin{aligned}
		\sum_{q=0}^{p-1} 
		(-1)^q \sum_{n=1}^\infty m_{{\rm cex},q,n}\log 
		\frac{\mu_{q,n}+\alpha_q}{\mu_{q,n}-\al_q} &= \sum_{q=0}^{p-1} (-1)^q 
		\frac{d}{ds}\Bl 
		\zeta_{\rm cex}(s;-\al_q)-\zeta_{\rm cex}(s;\al_q)\Br |_{s=0}\\
		&=\sum_{q=0}^{2p-1} (-1)^{q+1} 
		\frac{d}{ds}
		\zeta_{\rm cex}(s;\al_q) |_{s=0},
	\end{aligned}
\end{equation*}
where
\begin{equation*}
	\zeta_{\rm cex}(s;\pm\al_q):= \sum_{n=1}^{\infty} \frac{m_{{\rm 
				cex},q,n}}{ 
		(\mu_{q,n}\pm\al_q)^{-s}},
\end{equation*}
for $\Re(s)\gg 0$. 
These zeta functions have analytic continuations that are regular at $s=0$. 
\section{Some facts about zeta functions}
\label{s.Gen}

We use the notation $\Z$, $\R$ and $\C$ for the integer, real and complex numbers, respectively, $\Z_+=\{0,1,2,3,\ldots \}$ the non-negative integers and 
by $\N =\{1,2,3,\ldots\}$ the natural numbers. The residues 
\[
\Ru_{s=a} f(s) \qquad  \text{and} \qquad \Rz_{s=a}f(s)
\]
are the coefficients of $(z-a)^{-1}$ and $(z-a)^{0}$, respectively, in Laurent expansion of a meromorphic function $f$ at $a\in \C$.

\subsection{Some sums of Riemann zeta functions}
\label{ss.RiemannZeta}
 
We will denote by $\zeta_R(s)$ the Riemann zeta function, \ie
\[
\zeta_R(s) = \sum_{n=1}^{\infty} n^{-s}, \qquad \text{for}\; \Re(s)> 1.
\] 

Recall that,
\begin{equation}\label{eq.residue.zeta}
\Ru_{s=1}\zeta_R(s) = 1,\qquad\text{and}\qquad \Rz_{s=1}\zeta_R(s) = \gamma.
\end{equation}

Using the equality  \cite[Equation 9.532]{GR} 
\begin{equation}\label{eq.Sum.Zeta}
\sum_{k=2}^{\infty} (-1)^{k-1}\cdot\frac{z^k}{k} \cdot \zeta_R(k) = \ln 
\frac{e^{-\gamma\cdot z}}{\Gamma(1+z)}, \qquad\text{for}\qquad |z|<1,\;z\in\C.
\end{equation} 
It follows  that
\begin{equation}\label{eq.Sum.Zeta.Even-Odd}
\sum_{k=2}^\infty 
\frac{2^{-2k-1}}{2k+1} \cdot \zeta_R(2k+1) = \frac{1}{2}\Bl \ln 2-\gamma\Br,
\end{equation} and
\begin{equation}\label{eq.Sum.Zeta-1.Odd}
\sum_{k=1}^{\infty} \frac{2^{-2k-1}}{2k+1} \cdot (\zeta_R(2k+1)-1) = 
\frac{1}{2}\Bl1-\gamma+\ln\frac{2}{3} \Br.
\end{equation}

\begin{proposition} \label{Prop.Series.Estimate.Riemann}
	The series
	\begin{equation*}
	\sum_{j=1}^{\infty} \binom{-\frac{1}{2}}{j} 2^{-2j-1} \zeta_R(2j+1)
	\end{equation*} converges to a negative real number. Moreover, there are constants $c_1$ and $c_2$,
\begin{equation*}
\begin{aligned}
			c_1 &:= -\frac{1}{2} + \Bl \frac{1}{\sqrt{15}} + \frac{1}{\sqrt{17}}\Br + \frac{3}{4}\cdot \Bl \frac{1}{\sqrt{5}} - \frac{1}{\sqrt{3}}\Br,\\
		c_2 &:= -1 +2\cdot  \Bl \frac{1}{\sqrt{15}} + \frac{1}{\sqrt{17}}\Br + \frac{1}{2} \cdot \Bl \frac{1}{\sqrt{5}} - \frac{1}{\sqrt{3}}\Br,
\end{aligned}
\end{equation*}
such that
	\begin{equation}\label{eq.bound.sum.zeta}
c_1 \leq	\sum_{j=1}^{\infty} \binom{-\frac{1}{2}}{j} 2^{-2j-1} \zeta_R(2j+1) \leq c_2<0.
\end{equation}
\end{proposition}

\begin{proof}
We divide the sum in two, depending the parity of the index $j$ as follows
	\begin{equation*}
		\sum_{j=1}^\infty  \binom{-\frac{1}{2}}{j} 2^{-2j-1} \zeta_R(2j+1) = \sum_{j=1}^\infty  \binom{-\frac{1}{2}}{2j} 2^{-4j-1} \zeta_R(4j+1) + \sum_{j=1}^\infty  \binom{-\frac{1}{2}}{2j-1} 2^{-4j+1} \zeta_R(4j-1).
	\end{equation*}
Then, we observe that $\binom{-\frac{1}{2}}{j}$ is negative if $j$ is odd and positive if $j$ is even. Since $1<\zeta_R(2j+1)<2$, for all $j \in \N$,
\begin{equation*}
	 \sum_{j=1}^\infty \binom{-\frac{1}{2}}{2j} 2^{-4j-1} + 
\frac{3}{2} \sum_{j=1}^\infty \binom{-\frac{1}{2}}{2j-1} 2^{-4j+1}<\sum_{j=1}^\infty  \binom{-\frac{1}{2}}{j} 2^{-2j-1} \zeta_R(2j+1)
\end{equation*} 
and
\begin{equation*}
\sum_{j=1}^\infty  \binom{-\frac{1}{2}}{j} 2^{-2j-1} \zeta_R(2j+1)	< \sum_{j=1}^\infty \binom{-\frac{1}{2}}{2j} 2^{-4j} + 
	\sum_{j=1}^\infty \binom{-\frac{1}{2}}{2j-1} 2^{-4j+1}.
\end{equation*}
Now we use the power series of $(1+x)^{-\frac{1}{2}}$, for $|x|<1$, to obtain $c_1$ equal to 
\begin{equation}\label{eq.lower.bound}
	-\frac{1}{2} + \Bl \frac{1}{\sqrt{15}} + \frac{1}{\sqrt{17}}\Br + \frac{3}{4} \Bl \frac{1}{\sqrt{5}} - \frac{1}{\sqrt{3}}\Br,
\end{equation}
and $c_2$ equal to
\begin{equation}\label{eq.upper.bound}
	-1 +2\cdot  \Bl \frac{1}{\sqrt{15}} + \frac{1}{\sqrt{17}}\Br + \frac{1}{2} \Bl \frac{1}{\sqrt{5}} - \frac{1}{\sqrt{3}}\Br.
\end{equation}
The number $c_2$ is negative and the result is proved.
\end{proof}

\subsection{Some non-homogeneous zeta functions}

We will denote by $\zeta \Bl s; n^2+\frac{1}{4}\Br$ the following non-homogeneous zeta function
\begin{equation*}
\zeta \Bl s; n^2+\frac{1}{4}\Br:= \sum_{n=1}^{\infty} \Bl n^2+\frac{1}{4} \Br^{-s}, \qquad \text{for}\ \Re(s)>\frac{1}{2}.
\end{equation*}
This function has a meromorphic extension to the complex plane $\C$, regular at $s=0$, with a pole of order one at $\frac{1}{2}$ . The zero residue at $s=\frac{1}{2}$ is
\begin{equation*}
	\Rz_{s=\frac{1}{2}}\zeta \Bl s;n^2+\frac{1}{4} \Br = \gamma + \sum_{j=1}^\infty \binom{-\frac{1}{2}}{j} \cdot 4^{-j} \cdot\zeta_R(2\cdot j+1) ,
\end{equation*} 
and  the values of $\zeta\Bl s;n^2+\frac{1}{4}\Br $ and $\zeta'\Bl s;n^2+\frac{1}{4}\Br$ at $s=0$ are   (see \cite[Section 3.1]{Spr4})
\begin{equation}\label{eq.value.0.nonhomogeneous}
		\zeta\Bl 0;n^2+\frac{1}{4}\Br = -\frac{1}{2}, \qquad \text{and} \qquad
		\zeta'\Bl 0;n^2+\frac{1}{4}\Br= -2 \ln 2 - \ln \sinh\frac{\pi}{2}.
\end{equation}

\begin{proposition}\label{Prop.Series.Estimate.Nonhomogeneous}
	The series
	\begin{equation*}
		\sum_{j=1}^\infty \frac{2^{-2j}}{2j+1} \zeta \Bl \frac{2j+1}{2} ; n^2+\frac{1}{4} \Br 
	\end{equation*}
converges to a positive real number. Moreover, there are constants $c_3$ and $c_4$ such that
\begin{equation*}
0<c_3 \leq	\sum_{j=1}^\infty \frac{2^{-2j}}{2j+1} \zeta \Bl \frac{2j+1}{2} ; n^2+\frac{1}{4} \Br \leq c_4
\end{equation*}
with 
\begin{equation*}
	c_3:= \frac{1}{2}\cdot \Bl 1 - \gamma + \ln \frac{2}{3}  \Br, \qquad \text{and}\qquad c_4:= \frac{1}{2}\cdot \Bl \ln 2 - \gamma \Br.
\end{equation*}
\end{proposition}

\begin{proof}
For $s>1$,	we have
\begin{equation*}
		\frac{2^{-2s}}{2s}\Bl \zeta_R(2s)-1 \Br<\frac{2^{-2s}}{2s} \zeta\Bl s;n^2+\frac{1}{4}\Br< \frac{2^{-2s}}{2s}\zeta_R(2s),
\end{equation*} which implies that 
\begin{equation*}
	\sum_{j=1}^\infty \frac{2^{-2j-1}}{2j+1}\Bl \zeta_R(2j+1)-1 \Br<\sum_{j=1}^\infty \frac{2^{-2j-1}}{2j+1} \zeta\Bl \frac{2j+1}{2};n^2+\frac{1}{4}\Br<\sum_{j=1}^\infty \frac{2^{-2j-1}}{2j+1}\zeta_R(2j+1).
\end{equation*}
By equations \eqref{eq.Sum.Zeta.Even-Odd} and \eqref{eq.Sum.Zeta-1.Odd}, the result follows.
\end{proof}

We will work with  the shift by $\frac{1}{2}$ of the previous non-homogeneous zeta function, \ie
\begin{equation}\label{eq.nonhomogeneous.shift.definition}
\zeta_{\pm \frac{1}{2}} \Bl s; \sqrt{n^2+\frac{1}{4}}\Br:= \sum_{n=1}^{\infty} \Bl \Bl n^2+\frac{1}{4} \Br^{\frac{1}{2}}\pm \frac{1}{2}\Br ^{-s}, \qquad \text{for}\qquad \Re(s)>\frac{1}{2}.	
\end{equation}
We write this function as
\begin{equation}\label{eq.shif.perturb.shift}
	\zeta_{\pm\frac{1}{2}}\Bl s;\sqrt{n^2+\frac{1}{4}}\Br =\zeta \Bl\frac{s}{2}; n^2+\frac{1}{4}\Br +  \sum_{j=1}^{\infty} \binom{-s}{j} (\pm 2)^{-j} \zeta \Bl\frac{s+j}{2}; n^2+\frac{1}{4}\Br, 
\end{equation} 
for  $\Re(s)>\frac{1}{2}$.

\subsection{Some  Dirichlet series }

Consider a perturbed Dirichlet double series defined by
\begin{equation*}
	\zeta\Bl s;n^2+m^2+\frac{1}{4}\Br: = \sum_{n,m=1}^{\infty} \Bl n^2+m^2+ \frac{1}{4}\Br^{-s}, \qquad \text{for}\qquad \Re(s)> 1.
\end{equation*}
This zeta function has a meromorphic continuation to $\C$ with a pole at $s=\frac{1}{2}$. 

\begin{proposition}\label{prop.double.series}
	The double series 
	\begin{equation}
		\sum_{j=1}^\infty \frac{ 2^{-2j}}{2j+1} \zeta \Bl \frac{2j+1}{2}; n^2+m^2+\frac{1}{4} \Br
	\end{equation}
	converges to a positive real number. Moreover, there is a constant $c_5$ such that
	\begin{equation*}
		0\leq \sum_{j=1}^\infty \frac{ 2^{-2j}}{2j+1} \zeta \Bl \frac{2j+1}{2}; n^2+m^2+\frac{1}{4} \Br \leq c_5,
	\end{equation*}
	where $c_5 = \frac{9\Bl 2\sqrt{2} \tanh^{-1}\Bl \frac{1}{2\sqrt{2}} \Br -1\Br}{4\sqrt{2}}$.
\end{proposition}

\begin{proof}
	For $m,n\in \N$ we have
	\begin{equation*}
		2mn\leq m^2+n^2 \leq m^2+n^2+\frac{1}{4} \leq (m+1)^2(n+1)^2,
	\end{equation*}
	then for  $s>1$,
	\begin{equation*}
		\frac{2^{-2s}}{2s} \Bl \zeta_R(2s)-1\Br^2 \leq \frac{2^{-2s}}{2s} \zeta (s;n^2+m^2+\frac{1}{4}) \leq \frac{2^{-3s}}{2s}\zeta_R(2s)^{2}.
	\end{equation*}
	Consequently, taking $s=\frac{2j+1}{2}$ and summing in $j$ we obtain
	\begin{equation*}
		0 \leq  \sum_{j=1}^\infty \frac{ 2^{-2j}}{2j+1} \zeta \Bl \frac{2j+1}{2}; n^2+m^2+\frac{1}{4} \Br\leq \sum_{j=1}^{\infty} \frac{2^{-\frac{3}{2}(2j+1)}}{2j+1} \zeta_R(2j+1)^{2} \leq \frac{9}{4}\sum_{j=1}^{\infty} \frac{2^{-\frac{3}{2}(2j+1)}}{2j+1}.
	\end{equation*}
	Since,
	\begin{equation*}
		\sum_{j=1}^{\infty} \frac{2^{-\frac{3}{2}(2j+1)}}{2j+1} =  \tanh^{-1}\Bl \frac{1}{2\sqrt{2}} \Br - \frac{1}{2\sqrt{2}},
	\end{equation*}
	the result follows.
\end{proof}

As above, we consider the following shifted Dirichlet double series 
\begin{equation}\label{eq.shif.perturb.zeta.definition}
	\zeta_{\pm\frac{1}{2}}\Bl s;\sqrt{n^2+m^2+\frac{1}{4}}\Br:=\sum_{n,m=1}^{\infty} \Bl\Bl n^2+m^2 + \frac{1}{4} \Br^{\frac{1}{2}} \pm \frac{1}{2}\Br^{-s} , \;\; \text{for}\;\; \Re(s)>2.
\end{equation}
As before, we observed that
\begin{equation}\label{eq.shif.perturb.zeta}
	\zeta_{\pm\frac{1}{2}}\Bl s;\sqrt{n^2+m^2+\frac{1}{4}}\Br =\zeta \Bl\frac{s}{2}; n^2+m^2+\frac{1}{4}\Br +  \sum_{j=1}^{\infty} \binom{-s}{j} (\pm 2)^{-j} \zeta \Bl\frac{s+j}{2}; n^2+m^2+\frac{1}{4}\Br,
\end{equation} 
for $\Re(s)>2$.


\section{The anomaly term for  smooth manifolds}
\label{s.smothcae}
If the pseudomanifold $X$ is smooth, then it is a space with 
conical singularities where any point has a neighbourhood isometric to a cone over a sphere with dimension $\dim X - 1$. 
The interesting case here is when $\dim X$ is odd. Therefore, we will consider $\dim X = 
2p+1$ 
odd, $p\geq 1$ and the cross section of the cone, $W$, is a sphere of 
dimension $2p$, $p\geq 1$, \ie $W=S^{2p}$.  We fix a representation of $\pi_1(X)$ with rank one.

\subsection{Cell structure}
\label{ss.cell.S2p}
Let $S^{2p}$ be the sphere of dimension $2p$, $p\geq 1$, and radius one. $S^{2p}$ is a closed 
Riemannian manifold with metric induced by the Euclidean space $\R^{2p+1}$.  
The Intersection Reidemeister metric definition's requires a regular CW decomposition of of the cross section of the cone, however the combinatorial anomaly term depends only on the homology of the cross section and consequently we can use any CW decomposition to determine it.
A cell structure of this space is simple, with one $0$-dimensional cell, 
$e_{S^2p,0}$, and one $2p$-dimensional cell $e_{S^{2p},2p}$. We write
\[
S^{2p} = e_{S^{2p},0} \cup e_{S^{2p},2p}.
\] 

We may use this cells are representatives for the bases of the homology:
\[
H_0(S^{2p};\R) = \R[e_{S^{2p},0}] \qquad  \text{and} \qquad H_{2p}(S^{2p};\R) = \R[e_{S^{2p},2p}].
\] 

\subsection{Combinatorial anomaly term of $S^{2p}$}
\label{ss.Combterm}

We have that $\# TH_q(S^{2p};\Z) = 1 $, for all $q$. The set $\{1\}$, 
where $1$ denotes the constant $0$-form over $S^{2p}$ equal to $1$, is a 
basis for the 
harmonic forms of $S^{2p}$ in 
dimension $0$. The 
norm of this form is
\begin{equation*}
	\begin{aligned}
		\|1\|^2 &= \int_{S^{2p}} 1 \wedge \ast (1) = \Vol(S^{2p}) = \frac{2 
			\pi ^{\frac{2p+1}{2}}}{\Gamma(p+\frac{1}{2})} = 
		\frac{2^{p+1}\pi^p}{(2p-1)!!}.
	\end{aligned}
\end{equation*}
Then an orthonormal base of the harmonic forms in dimension zero is 
$\{\frac{1}{\|1\|}\}$ and 
\begin{equation*}
	\tilde{\mathcal{A}}_0(1/\|1\|) = 
	\frac{(2^{p+1}\pi^p)^{\frac{1}{2}}}{((2p-1)!!)^{\frac{1}{2}}} \cdot e_{S^{2p},0}.
\end{equation*}
Therefore,
\begin{equation*}
	A_{\rm comb, \mf}(S^{2p}) = - 
	\frac{1}{2}\log \frac{2^{p+1} \pi^p}{(2p-1)!!}.
\end{equation*}

\subsection{Eigenvalues of the Hodge-Laplace operator on the even dimensional 
	euclidean sphere}
\label{ss.Eigenvalues.S2p}
The complete 
system of eigenvalues with multiplicities of the Laplacian on $S^{2p}$  is well known (see 
\cite[pag. 113]{WY} and \cite[Theorem 6.8]{IT}). The eigenvalues associated 
to co-exact eigenforms in dimension $q$ are 
\begin{equation*}
	\begin{aligned}
		&\lambda_{q,n} := (n+q)(n+2p-1-q), & 0\leq q \leq p-1, n\geq 1,
	\end{aligned}
\end{equation*}
with multiplicity
\begin{equation}\label{eq.multS2p}
	m_{q,n}:=m_{{\rm cex},q,n} =
	\frac{2n-1+2p}{n+2p-1-q}\binom{n-1+2p}{n+q}\binom{q+n-1}{n-1}.
\end{equation}

Using the notation of the Section \ref{ss.CMtheorem}, we will manipulate the 
terms $\mu_{q,n}$ and $m_{q,n}$ in order to simplify
\Eqref{eq.anomalyanaly.3}. If we complete the 
squares we obtain
\begin{equation*}
	\mu_{q,n}^2 = \lambda_{q,n} + \al_q^2=\left(n+q+ p -q 
	-\frac{1}{2}\right)^2 = \left(n+p - 
	\frac{1}{2}\right)^2.
\end{equation*}
Consequently, $\mu_{q,n} = n+p - \frac{1}{2}$, for $n\geq 1$ and $0\leq q 
\leq p-1$ and this implies that
\begin{equation*}
	\mu_{q,n} \pm \al_q= 
	n+p - \frac{1}{2} + \frac{1}{2}+ q - p = n+q. 
\end{equation*}
In relation to the multiplicity, we expand \Eqref{eq.multS2p} to 
obtain
\begin{equation*}
	\begin{aligned}
		m_{q,n} 
		&= \frac{1}{(2p-1)!} \binom{2p-1}{q} (2n-1+2p) 
		\prod_{\stackrel{j=0}{j\not = q,2p-1-q}}^{2p-1} (n+j).
	\end{aligned}
\end{equation*}
We observe the particular case
\begin{equation}\label{eq.mqn-q}
	m_{q,n-q} = \frac{1}{(2p-1)!} \binom{2p-1}{q} (2n-1-2q+2p) 
	\prod_{\stackrel{j=0}{j\not = q,2p-1-q}}^{2p-1} (n+j-q).
\end{equation}

The way that we wrote the multiplicity shows that we can use the elementary 
symmetric polynomial to rewrite the 
multiplicities as polynomials in $n$ as follows. Define the symmetric 
polynomial associated to the sequence $N^q$ by 
\begin{equation}\label{eq.prodn-q}
	S(x;N^q):=\prod_{\stackrel{j=0}{j\not = q,2p-1-q}}^{2p-1} (x+j-q)  = 
	\sum_{l=0}^{2p-2} e_{2p-2-l}(N^q) x^l,
\end{equation}
where
\[
N^q = (-q,1-q,2-q,\ldots,\widehat{q-q},\ldots,\widehat{2p-1-q - q},\ldots, 
2p-1-q),
\] 
$e_{2p-2-l}(N^q)$ is the elementary symmetric polynomial of degree 
$2p-2-l$ of $N^q$ and the $\widehat{\cdot}$ means that these values are omitted.
The roots of $S(x;N^q)$ are the integers numbers 
from $q+1-2p$ to $q$ except for 
$x=-2p+1+2q$ and $x=0$, 
where 
$q=0,\ldots, 2p-1$. The values of $S(x;N^q)$ at those two points are 
\begin{equation}\label{eq.nonzerop_-}
	\begin{aligned}
		S(-2p+1+2q;N^q)=S(0;N^q) &= e_{2p-2}(N^q)
		= (-1)^q 
		q!\frac{(2p-1-q)!}{(2p-1-2q)}.
	\end{aligned}
\end{equation}
Then, the multiplicity can be written as
\[
m_{q,n-q} = \frac{1}{(2p-1)!} \binom{2p-1}{q}\ (2n-1-2q+2p) \ S(n;N^q).
\]

Consider the zeta function defined in Section \ref{ss.CMtheorem},
\begin{equation*}
	\zeta_{\rm cex}(s;\al_q)= \sum_{n=1}^\infty 
	\frac{m_{q,n} }{(\mu_{q,n}+\al_q)^{-s}}= \sum_{n=1}^\infty \frac{m_{q,n}}{(n+q)^s} .
\end{equation*} 
If we change the starting number of the sum we obtain,
\begin{equation}\label{eq.zetasphere}
	\begin{aligned}
		\zeta_{\rm cex}(s;\al_q) &
		= \sum_{n=q+1}^\infty \frac{m_{q,n-q}}{n^s}.
	\end{aligned}
\end{equation}
We use \Eqref{eq.prodn-q} in 
\Eqref{eq.mqn-q},  and  the last in \Eqref{eq.zetasphere}, to obtain 
\begin{equation*}
	\begin{aligned}
		\zeta_{\rm cex}(s;\al_q) 
		&=\frac{2}{(2p-1)!} \binom{2p-1}{q}   
		\sum_{l=0}^{2p-2}e_{2p-2-l}(N^q) \left(\zeta_R(s-l-1) - 
		\sum_{n=1}^{q} n^{1+l-s}\right)\\
		&-\frac{2\al_q}{(2p-1)!} \binom{2p-1}{q}   
		\sum_{l=0}^{2p-2}e_{2p-2-l}(N^q) \left(\zeta_R(s-l) - \sum_{n=1}^{q} 
		n^{l-s}\right).
	\end{aligned}
\end{equation*}
With the information about the zeros of $S(x;N^q)$ we obtain for $q\leq p-1$,
\begin{equation*}
	\begin{aligned}
		\zeta_{\rm cex}(s;\al_q) &=\frac{2}{(2p-1)!} \binom{2p-1}{q}   
		\sum_{l=0}^{2p-2}e_{2p-2-l}(N^q) \left(\zeta_R(s-l-1) 
		-\al_q\zeta_R(s-l)\right),
	\end{aligned}
\end{equation*}
and applying \Eqref{eq.nonzerop_-} for $q\geq p$,
\begin{equation*}
	\begin{aligned}
		\zeta_{\rm cex}(s;\al_q) &=\frac{2}{(2p-1)!} 
		\binom{2p-1}{q}   
		\sum_{l=0}^{2p-2}e_{2p-2-l}(N^q) \left(\zeta_R(s-l-1) 
		+\al_q\zeta_R(s-l)\right)\\
		&-(-1)^q(-2p+1+2q)^{-s} .
	\end{aligned}
\end{equation*}

Remember that our goal is to determine the value of \Eqref{eq.anomalyanaly.3}, 
so we multiply by $(-1)^{q+1}$ the last two equations and then we sum up 
in 
$q$, 
from $0$ to $2p-1$, to obtain
{\small \begin{equation}\label{eq.sum.zeta.muqn}
	\begin{aligned}
		\frac{2}{(2p-1)!} 
		\sum_{q=0}^{2p-1}(-1)^{q+1}&\binom{2p-1}{q}\sum_{l=0}^{2p-2}e_{2p-2-l}(N^q)
		\left(\zeta_R(s-l-1) 
		-\al_q\zeta_R(s-l)\right)\\
		&-\sum_{q=p}^{2p-1} (-2p+1+2q)^{-s}.
	\end{aligned}
\end{equation}}

\noindent Note that $e_{2p-2-l}(N^q)$ and $\al_q e_{2p-2-l}(N^q)$ are a polynomials in 
$q$ with degree $2p-2-l$ and $2p-1-l$, respectively, 
and $l=0,\ldots, p-2$. Therefore the maximum degrees are $2p-2$ and $2p-1$, 
respectively. The 
equation
\cite[0.154-6]{GR} says that
\[
\sum_{q=0}^{2p-1} (-1)^q \binom{2p-1}{q} q^{k} = 0,
\] 
for all $0\leq k \leq 2p-2$. Then 
\[
\sum_{q=0}^{2p-1} (-1)^{q+1}  \binom{2p-1}{q} e_{2p-2-l}(N^q)  =
\sum_{q=0}^{2p-1} (-1)^{q+1}  \binom{2p-1}{q} \alpha_q e_{2p-2-l}(N^q)=0,
\] for $l=0,\ldots,p-2,$ in the first sum, and $l=1,\ldots,p-2,$ in the second 
sum. It remains to deal with the case $l=0$ in the second sum, \ie
\[
\frac{1}{(2p-1)!} \sum_{q=0}^{2p-1}(-1)^{q+1}\binom{2p-1}{q} \al_q 
e_{2p-2}(N^q).
\]
However, using \Eqref{eq.nonzerop_-} this sum is equal to $2p$.
Consequently,  \Eqref{eq.sum.zeta.muqn} reduces to
\begin{eqnarray*}
	\sum_{q=0}^{2p-1}(-1)^{q+1} \zeta_{\rm cex}(s;\al_q) = 	2 p \zeta_R(s) - 
	\sum_{q=p}^{2p-1} (-2p+1+2q)^{-s}.
\end{eqnarray*}

\subsection{Analytic anomaly term}
\label{ss.Analyticterm}

By the manipulations presented in Section \ref{ss.Eigenvalues.S2p} the 
 analytic anomaly term  in the smooth case is
\[
\begin{aligned}
	A_{\rm analy}(S^{2p}) &= -\log(2p-1)!! +\frac{1}{2} \log 2
	+  
	\frac{1}{2} \frac{d}{ds} \left(\sum_{q=0}^{2p-1}(-1)^{q+1} 
	\zeta(s;\mu_{q,n}+\al_q) \right)|_{s=0}\\	
	&= -\log (2p-1)!!+\frac{1}{2}\log 2 -p \zeta'_R(0) +\frac{1}{2} 
	\sum_{q=p}^{2p-1} 
	\log(-2p+1+2q)\\
	&= \frac{1}{2} \log \frac{2^{p+1}\pi^p}{(2p-1)!!}.
\end{aligned}
\]

\subsection{Cheeger-M\"uller Theorem}
\label{ss.CMsmooth}

Sections \ref{ss.Combterm} and \ref{ss.Analyticterm} prove the 
following corollary of Theorem \ref{theo.CMSm}.

\begin{proposition}
	If $X$ is a smooth manifold then Theorem \ref{theo.CMSm} is the 
	classical Cheeger-M\"uller theorem.
\end{proposition}

\section{The anomaly for the cone over a torus}
\label{s.Nontriviality}
  Let $T=S^1\times S^1$ be the bi-dimensional torus, where $S^1$ denotes the 
circle with radius one. This 
is a closed orientable Riemannian manifold with a product metric $g = 
g_{S^1\times \bullet}+ g_{\bullet \times S^1} = d\theta^2 + d\varphi^2$. 
We will study the anomaly terms, $A_{\rm analy} (T)$ and $A_{\rm comb,\mf} (T)$\footnote{As in Section \ref{s.smothcae}, we fix a representation of $\pi_1(X)$ with rank one.}. On the topological side we are able to determine precisely the value of the anomaly, however on the analytic side we will just give an estimate. This will be enough to guarantee that the sum $A_{\rm analy} (T)+A_{\rm comb,\mf} (T)$ is non zero. 

\subsection{Cell structure of Torus}
\label{ss.cell.Torus}

As observed in Subsection \ref{ss.cell.S2p} we can use any cell decomposition to determine the anomaly combinatorial term. A 
cell structure of $T$ is build using the cell structure of $S^1$, \ie one 
cell in dimension zero, $e_{T,0} = e_{S^1,0} \times e_{S^1_0}$, two cell in 
dimension one, $ \{e_{S^1,1}\times e_{S^1,0},\;\; e_{S^1,0}\times 
e_{S^1,1}\}$ and one cell in dimension two $e_{T,2} = e_{S^1,0}\times 
e_{S^1,1}$. These cells represent generators of the homology groups,
\begin{equation*}
	\begin{aligned}
		H_0(T;\R) &= \R [e_{S^1,0}\times e_{S^1,0}], \;\; H_1(T;\R) = \R 
		[e_{S^1,1}\times e_{S^1,0}] \oplus \R[e_{S^1,0}\times e_{S^1,1}],\\
		H_2(T;\R)&= \R [e_{S^1,1} \times e_{S^1,1}].
	\end{aligned}
\end{equation*}

\subsection{Combinatorial anomaly term of $T$}
\label{ss.CombtermTorus}

Similar to the sphere case, the cardinality of the torsion subgroup is $\# 
TH_q(T;\Z) = 1$, for all $q$. It remains to determine 
$\det(\tilde{\mathcal{A}}_0(\tilde{h}_0)/n_0)$ and $\det(\tilde{\mathcal{A}}_1(\tilde{h}_1)/n_1)$.
In dimension $0$, the constant eigenform is a basis for the harmonic forms, so 
we consider this form constant equal to $1$. Then 
\begin{equation*}
\tilde {\mathcal{A}}_0\Bl 
1/\|1\|\Br =	\tilde {\mathcal{A}}_0\Bl 
	1/\sqrt{\Vol (T)}\Br = \sqrt{\Vol (T)}e_{T,0} = 2\pi \cdot e_{T,0}. 
\end{equation*}
The forms $d\theta$ and $d\varphi$ are orthogonal generators 
of the harmonic forms in dimension one, then we calculate
\begin{equation*}
	\|d\theta\| = \|d\varphi\| = \sqrt{\Vol (T)} = 2\pi.
\end{equation*}
Applying $\tilde{\mathcal{A}}_1$ in $d\theta/\|d\theta\|$ and 
$d\varphi/\|d\varphi\|$, we have
\begin{equation*}
	\tilde{\mathcal{A}}_1\Bl \frac{d\theta}{\|d\theta\|} \Br = 
	e_{S^1,0}\times e_{S^1,1},\qquad
	\tilde{\mathcal{A}}_1\Bl \frac{d\varphi}{\|d\varphi\|} \Br = 
	e_{S^1,1}\times e_{S^1,0}.
\end{equation*}
With this information, the 
combinatorial anomaly term of the torus reads
\begin{equation*}
	A_{\rm comb,\mf}(T) = -\log 2\pi.
\end{equation*}

\subsection{Eigenvalues of the Hodge-Laplace operator on the bidimensional 
	torus}
\label{ss.Eigenvalues.Torus}

The eigenvalues of the Hodge-Laplace operator can obtained from the 
eigenvalues of $S^1$ using the product rule. We are interested in the eigenvalues of the 
co-exact eigenforms. These eigenvalues are $n^2+m^2$ for $n\geq 1$ and $m \geq 
0$ with multiplicity equal to $4$. Observe that in this case $p=1$, $\al_0 = 
-\frac{1}{2}$, $\mu_{0,n,m} = \sqrt{n^2+m^2+\frac{1}{4}}$.

The zeta functions of interest are
\begin{equation*}
	\begin{aligned}
		\zeta_{\rm cex}\Bl s;\pm\frac{1}{2}\Br &= 4 \sum_{n=1}^{\infty} 
		\sum_{m=1}^{\infty} \frac{1}{\Bl \Bl n^2+m^2+\frac{1}{4}\Br^{\frac{1}{2}} \pm 
			\frac{1}{2}\Br^{s}}+4 \sum_{n=1}^{\infty} \frac{1}{\Bl 
			\Bl n^2+\frac{1}{4}\Br^{\frac{1}{2}} \pm \frac{1}{2}\Br^{s}}\\
		&=4 \cdot \zeta_{\pm\frac{1}{2}}\Bl s;\sqrt{n^2+m^2 +\frac{1}{4}}\Br + 4 \cdot 
		\zeta_{\pm\frac{1}{2}}\Bl s;\sqrt{n^2+\frac{1}{4}}\Br,
	\end{aligned}
\end{equation*}
and the value that we are looking for is
\begin{equation*}
	\begin{aligned}
		\zeta'_{\rm cex}\Bl0;\frac{1}{2}\Br - \zeta'_{\rm 
			cex}\Bl0;-\frac{1}{2}\Br &= 4\Bl \zeta'_{\frac{1}{2}}\Bl 
		0;\sqrt{n^2+m^2 +\frac{1}{4}}\Br -\zeta'_{-\frac{1}{2}}\Bl 
		0;\sqrt{n^2+m^2 +\frac{1}{4}}\Br \Br\\
		&+  
		4\Bl\zeta'_{\frac{1}{2}}\Bl0;\sqrt{n^2+\frac{1}{4}}\Br -  
		\zeta'_{-\frac{1}{2}}\Bl0;\sqrt{n^2+\frac{1}{4}}\Br \Br.
	\end{aligned}
\end{equation*}
We define functions $Z_1(s)$ and $Z_2(s)$ by
\begin{equation*}
	\begin{aligned}
		Z_1(s):=& \zeta'_{\frac{1}{2}}\Bl 
		s;\sqrt{n^2+m^2 +\frac{1}{4}}\Br-\zeta'_{-\frac{1}{2}}\Bl 
		s;\sqrt{n^2+m^2 +\frac{1}{4}}\Br,\\
		Z_2(s):=& \zeta'_{\frac{1}{2}}\Bl s;\sqrt{n^2+\frac{1}{4}}\Br -  
		\zeta'_{-\frac{1}{2}}\Bl s;\sqrt{n^2+\frac{1}{4}}\Br.
	\end{aligned}
\end{equation*}

\subsection{Analytic anomaly term}
\label{ss.AnalytermTorus}

The  analytic anomaly term, in the case that $T$ is the cross section of the cone, is
\begin{equation*}
\begin{aligned}
		A_{\rm analy} (T) &= \frac{1}{2} \frac{d}{ds}\Bl 
	\zeta_{\rm cex}\Bl s;\frac{1}{2}\Br-\zeta_{\rm cex}\Bl 
	s;-\frac{1}{2}\Br\Br |_{s=0}\\
	&=2 \cdot \Bl Z_1(0) +  Z_2(0) \Br.
\end{aligned}
\end{equation*}
We will estimate the value of $A_{\rm analy}(T)$.

\subsubsection{The term $Z_1(0)$}
\label{ss.AnalytermTorus.Z_1}

We consider \Eqref{eq.shif.perturb.zeta} and make the difference between the zeta functions to write , 
\begin{equation*}
	Z_1(s) = \binom{-s}{1} \zeta \Bl \frac{s+1}{2}; n^2+m^2+\frac{1}{4} \Br + 2 \sum_{j=1}^\infty \binom{-s}{2j+1} 2^{-2j-1} \zeta \Bl \frac{s+2j+1}{2}; n^2+m^2+\frac{1}{4} \Br, 
\end{equation*}
and then, using \Eqref{eq.Binomial}, we obtain
\begin{equation*}
	Z_1(0) = - \Rz_{s=\frac{1}{2}} \zeta \Bl s ; n^2+m^2+\frac{1}{4} \Br - 2 \sum_{j=1}^\infty \frac{ 2^{-2j-1}}{2j+1} \zeta \Bl \frac{2j+1}{2}; n^2+m^2+\frac{1}{4} \Br.
\end{equation*}

First we estimate the residue term. Using  Mellin transform, Poisson summation formula and \Eqref{eq.Bessel.transf} we obtain, for $\Re(s)>2$,
\begin{equation*}
\begin{aligned}
		\zeta\Bl s;n^2+m^2+\frac{1}{4} \Br 
		&=2 \frac{\pi^s}{\Gamma(s)} \sum_{n,m=1}^\infty \Bl \frac{m^2}{ n^2+\frac{1}{4}} \Br^{\frac{s-\frac{1}{2}}{2}} K_{s-\frac{1}{2}} \Bl 2 \pi \cdot  m \cdot \Bl n^2
		+\frac{1}{4} \Br^{\frac{1}{2}} \Br \\
		&+ \frac{1}{2}\frac{\sqrt{\pi}}{\Gamma(s)} \cdot \Gamma\Bl s-\frac{1}{2}\Br \cdot \zeta\Bl s-\frac{1}{2}; n^2+\frac{1}{4}\Br-\frac{1}{2} \zeta \Bl s; n^2+\frac{1}{4}\Br,
\end{aligned}
\end{equation*}
where $K_{\nu}(x)$ is the modified Bessel function with order $\nu$ and argument $x$.
The poles at $s=\frac{1}{2}$ of $\frac{1}{2}\frac{\sqrt{\pi}}{\Gamma(s)} \cdot \Gamma\Bl s-\frac{1}{2}\Br \cdot \zeta\Bl s-\frac{1}{2}; n^2+\frac{1}{4}\Br$ and 
$\frac{1}{2} \zeta \Bl s; n^2+\frac{1}{4}\Br$ cancel and then we obtain the analytic continuation of $\zeta\Bl s;n^2+m^2+\frac{1}{4} \Br $ in a neighbourhood of $s=\frac{1}{2}$. 
Moreover,
\begin{equation*}
\begin{aligned}
		\Rz_{s=\frac{1}{2}}	\zeta\Bl s;n^2+m^2+\frac{1}{4} \Br  &= 2 \sum_{n,m=1}^{\infty} K_0 \Bl 2\cdot m \cdot\pi  \Bl n^2+\frac{1}{4}  \Br^{\frac{1}{2}} \Br -\frac{3}{2}\ln 2- \frac{1}{2} \ln \sinh \frac{\pi}{2}\\
		&+ \frac{\gamma}{2} +\sum_{j=1}^\infty \binom{-\frac{1}{2}}{j} \zeta_R(2j+1) 2^{-2j-1}.
\end{aligned}
\end{equation*}
Using the results of Appendix \ref{Appendix.Mod.Bessel}, the first sum satisfies the following inequality
\begin{equation}\label{eq.Sum.Bessel.Ine}
	0 < \sum_{n,m=1}^{\infty} K_0 \Bl 2\cdot m \cdot\pi  \Bl n^2+\frac{1}{4}  \Br^{\frac{1}{2}} \Br < \frac{1}{2} \sum_{n,m=1}^\infty e^{-2\cdot \pi\cdot n\cdot m} = \frac{e}{2}\Bl \frac{e^{2\cdot \pi}}{e^{2\cdot\pi}-1} - 1 \Br.
\end{equation}
From Propositions \ref{Prop.Series.Estimate.Riemann} and \ref{Prop.Series.Estimate.Nonhomogeneous}, and what it was presented before we prove the following estimate
\begin{equation*}
e\Bl 1-\frac{e^{2\cdot \pi}}{e^{2\cdot\pi}-1}\Br  +\frac{1}{2} \ln\Bl 8\sinh\frac{\pi}{2}\Br - \frac{\gamma}{2}  -c_2  -c_5	\leq Z_1(0) \leq  \frac{1}{2} \ln\Bl 8\sinh\frac{\pi}{2}\Br - \frac{\gamma }{2} -c_1 .
\end{equation*}

\subsubsection{The term $Z_2(0)$}
\label{ss.AnalytermTorus.Z_2}

 By \Eqref{eq.shif.perturb.shift}  we have
\[
Z_2(0) = - \Rz_{s=\frac{1}{2}} \zeta\Bl s; n^2+\frac{1}{4}\Br - \sum_{j=1}^\infty \frac{2^{-2j}}{2j+1} \zeta\Bl \frac{2j+1}{2};n^2+\frac{1}{4} \Br .
\]
The residue term can be estimate as follow, using \Eqref{eq.Binomial}, we can write
\[
\Rz_{s=\frac{1}{2}} \zeta\Bl s; n^2+\frac{1}{4}\Br = \gamma + \sum_{j=1}^\infty \binom{-\frac{1}{2}}{j} 2^{-2j} \zeta_R(2j+1).
\] 
By Proposition \ref{Prop.Series.Estimate.Riemann} we obtain
\begin{equation*}
	c_1+\gamma \leq \Rz_{s=\frac{1}{2}} \zeta\Bl s; n^2+\frac{1}{4}\Br \leq c_2+\gamma.
\end{equation*}
With this equation and Proposition \ref{Prop.Series.Estimate.Nonhomogeneous}, a estimate for $Z_2(0)$ is
\begin{equation*}
-c_2-c_4-\gamma \leq Z_2(0) \leq  -c_1-c_3-\gamma.
\end{equation*}

\subsubsection{Bound for the analytic anomaly term of $T$}
\label{ss.AnalytermTorus.bound}
Collecting the information from Subsections \ref{ss.AnalytermTorus.Z_1} and \ref{ss.AnalytermTorus.Z_2}, an estimate for $A_{\rm analy}(T)$ is 
\begin{equation*}
\begin{aligned}
A \leq A_{\rm analy} (T) \leq  B,
\end{aligned}
\end{equation*}
where
\begin{equation*}
	\begin{aligned}
		A&= 4 -8\cdot  \Bl \frac{1}{\sqrt{15}} + \frac{1}{\sqrt{17}}\Br - 2 \cdot \Bl \frac{1}{\sqrt{5}} - \frac{1}{\sqrt{3}}\Br + \ln\Bl 2\sinh\Bl\frac{\pi}{2}\Br\Br - \gamma\\
		&+2e\Bl 1-\frac{e^{2\cdot \pi}}{e^{2\cdot\pi}-1}\Br-\frac{9\Bl 2\sqrt{2} \tanh^{-1}\Bl \frac{1}{2\sqrt{2}} \Br -1\Br}{2\sqrt{2}},\\
		B&=-4 \Bl \frac{1}{\sqrt{15}} + \frac{1}{\sqrt{17}}\Br - 3 \Bl \frac{1}{\sqrt{5}} - \frac{1}{\sqrt{3}}\Br+\ln\Bl 18\sinh\Bl\frac{\pi}{2}\Br\Br -\gamma.
	\end{aligned}
\end{equation*}

\subsection{The anomaly on Cheeger-M\"uller Theorem }

Using the estimates presented before we obtain
\begin{equation*}
-\frac{4}{5}<A-\ln(2\pi) \leq	A_{\rm analy} (T) +A_{\rm comb,\mf} (T) \leq  B-\ln(2\pi)<-\frac{1}{4}<0.
\end{equation*}

\begin{proposition}
	If $X$ is a pseudomanifold with a conical singularity such that the section of the cone is a bi-dimensional torus then there is a non trivial contribution from the singularity in the quotient of the intersection Reidemeister  metric and the Ray-Singer metric defined on intersection homology line bundle with middle perversity.
\end{proposition}

\appendix
\section{Some formulas}\label{Appendix}

\subsection*{Binomial coefficient}
The binomial coefficient $\binom{-s}{j}$ with $s\in \C$ has the following expansion near $s=0$,
\begin{equation}\label{eq.Binomial}
	\binom{-s}{j} = \frac{(-1)^j}{j} s + \frac{(-1)^j}{j} \Bl \psi(j) + \gamma \Br s^2 + O(s^3).
\end{equation}




\subsection*{Modified Bessel function of second type}\label{Appendix.Mod.Bessel}

The modified Bessel function of second type with order $\nu$ and argument $z$ is a solution of the following differential equation
\begin{equation*}
	\frac{d^2 w}{dz^2} + \frac{1}{z} \frac{dw}{dz} - \Bl 1+ \frac{\nu^2}{z^2}\Br w = 0,
\end{equation*}
and it denoted by $K_\nu(z)$. When $\Re(\nu)\geq 0$, this function has a singularity at $z=0$.  $K_\nu (z)$ is real when $\nu$ is real and $z$ is positive. If $\nu\geq 0$  is fixed then $K_\nu (z)$ is positive, decreasing ($z\in (0,\infty)$) and has the following behavior for $z\to \infty$
\begin{equation*}
	K_\nu(z) = \sqrt{\frac{\pi}{2z}} e^{-z} (1+O(z^{-1})).
\end{equation*} If $z>0$ is fixed then $K_\nu(z)$ is increasing for $\nu \geq 0$, moreover,
\begin{equation*}
	K_0(z)< K_{\frac{1}{2}} (z) = \sqrt{\frac{\pi}{2z}} e^{-z}, \qquad z\in(0,\infty).
\end{equation*}

The next identity can be found in \cite[Eq. 3.471-9]{GR}
\begin{equation}\label{eq.Bessel.transf}
	\int_0^\infty t^{s-1} e^{-\frac{\beta}{t} -zt } dt = 2 \Bl \frac{\beta}{z}\Br^{\frac{s}{2}} K_{s} (2\sqrt{\beta z}), \qquad \text{for}\;\; \Re(\beta)>0, \; \Re(z)>0.
\end{equation}

\bibliography{HarSprBibliography}
\bibliographystyle{amsalpha}


\end{document}